\documentclass[12pt]{article}

\usepackage[usenames]{color} 
\usepackage{amssymb} 
\usepackage{amsmath} 
\usepackage[utf8]{inputenc} 
\usepackage{comment}
\usepackage{hyperref}
\usepackage{bbm}
\usepackage{epsfig}
\usepackage{subfigure}
\usepackage{enumerate}

\includecomment{problem}

\includecomment{solution}

\def\hf {\mbox{$\frac{1}{2}$}}
\def\ed { \stackrel{d}{=} }
\def\E { {\mathbb{E}}}
\def\P { {\mathbb{P}}}
\def\R { {\mathbb{R}}}

\def\Z { {\mathbb{Z}}}

\def\CC{ {\mathcal{C} }}

\def\1{\mathbbm{1} }
\def\Le{Lévy }
\def\LK{Lévy-Khinchine }
\def\LI{Lévy-It\^o }

\newcommand{\eps}{\varepsilon}

\newcommand{\Tsvar}{ (T(s), s > 0)}

\newcommand{\ssarate}{rate }

\def \endprf{\hfill {\vrule height6pt width6pt depth0pt}\medskip} 
\newenvironment{proof}{\noindent {\bf Proof} }{\endprf\par}

\newtheorem{theorem}{Theorem}[section]
\newtheorem{lemma}[theorem]{Lemma}
\newtheorem{proposition}[theorem]{Proposition}
\newtheorem{corollary}[theorem]{Corollary}

\newtheorem{example}[theorem]{Example}

\numberwithin{equation}{section}

\newtheorem{innercustomgeneric}{\customgenericname}
\providecommand{\customgenericname}{}
\newcommand{\newcustomtheorem}[2]{%
  \newenvironment{#1}[1]
  {%
   \renewcommand\customgenericname{#2}%
   \renewcommand\theinnercustomgeneric{##1}%
   \innercustomgeneric
  }
  {\endinnercustomgeneric}
}

\newcustomtheorem{customthm}{Theorem}
\newcustomtheorem{customlemma}{Lemma}

\title{The range of a self-similar additive gamma process is a scale invariant Poisson point process}
\author{Jim Pitman and Zhiyi You}
\date{\today}

\begin{document}
	
\maketitle

\section*{Abstract}

It is shown that for a non-decreasing self-similar stochastic process $T$ with independent increments,
the range of $T$ forms a Poisson point process with $\sigma$-finite intensity if and only if the one-dimensional distribution of $T(1)$ is of the gamma type. This follows from a general hold-jump description of such processes $T$, and implies the known result that the spacings between consecutive points of a scale invariant Poisson point process, with intensity $\theta x^{-1} dx$, are the points of another scale invariant Poisson point process with the same intensity.


\section*{Key words}

scale invariant Poisson spacings, Poisson point processes, self-similar processes, Sato processes, range, records, extremal processes

\section{Introduction}\label{sec:intro}
A point process $N_X(B) = \# \{z \in \Z : X_z \in B\}$, counting numbers of points in subintervals $B$ of the positive half line $\R_+ := (0, \infty)$, for some indexing of these points $X_z$ by $z$ in the set of integers $\Z$, is called {\em scale invariant} if for each $c > 0$ the point process $N_{cX}$, counting the scaled points $\{c X_z: z \in \Z\}$, has the same finite-dimensional distributions as $N_X$, as $B$ varies over subintervals of $\R_+$.
Assuming the point process $N_X$ is {\em simple}, meaning the points $X_z$ are all distinct, the point process $N_X$ is then regarded as encoding the random countable set 
\begin{equation}
    range(X_z, \in \Z):= \{ X_z, \in \Z \}.
\end{equation}
So the identity in distribution of simple point processes $N_X$ and $N_{cX}$ may be indicated by the notation
\begin{equation}
    range ( X_z, z \in \Z ) \ed range ( c  X_z, z \in \Z ).
\end{equation}

It is well known that a {\em Poisson point process} ({\em PPP}) on $\R_+$ is scale invariant if and only if its intensity measure is $\theta x^{-1}dx$ for some $\theta \ge 0$, when the process is called a $PPP(\theta x^{-1} dx)$, or a {\em scale invariant Poisson point process} with {\em rate} $\theta$.
So the parameter $\theta$ is the intensity of such a point process relative to the scale invariant measure $x^{-1}dx$ on the positive half-line, and for $0 < a < b < \infty$, the number of points in $(a, b)$ has a Poisson distribution with mean $\theta \log(b / a)$.

The following result on scale invariant Poisson spacings arises in various contexts.
The case $\theta = 1$ is contained in the theory of records and extremal processes, first developed by Dwass \cite{dwass1964extremal, dwass1966extremal, dwass1974extremal} and further studied by Resnick and Rubinovitch \cite{resnick1973structure}, and Shorrock \cite{shorrock1974discrete}.
The formulation for general $\theta > 0$ is due to Arratia \cite{arratia1998central, arratia2002amount}, who sketched a proof which was later detailed by Arratia, Barbour and Tavar\'e \cite[Section 7]{arratia2006tale}.

\begin{theorem}[Scale invariant Poisson spacings]\label{thm:sips}
    Fix $\theta > 0$. Let $(T_z, z \in \Z)$,  with $T_{z} < T_{z+1}$ for all $z \in \Z$, be an exhaustive listing of the points of a scale invariant PPP with rate $\theta$.
    Then 
    \begin{equation}\label{eqn:sips}
        range (T_{z+1} - T_z, z \in \Z) \ed range (T_z, z \in \Z).
    \end{equation}
\end{theorem}
Less formally, the theorem states:
\begin{itemize}
    \item {\em the spacings between consecutive points of a scale invariant Poisson point process on the positive half-line are the points of another scale invariant Poisson point process with the same rate}.
\end{itemize}

This article places Theorem \ref{thm:sips} in the broader context of stochastic processes $T = \Tsvar$ which are {\em self-similar additive} ({\em SSA}), meaning that
\begin{itemize}
    \item $T$ is {\em self-similar}: for all $c > 0$,
    \begin{equation}\label{eqn:1selfsimilar}
        (T({cs}), s > 0) \ed (c T(s), s > 0);
    \end{equation}
    \item $T$ is {\em additive}: meaning $T$ has independent increments, and $T$ is stochastically continuous with càdlàg paths starting at $0$
    \begin{equation}
        \lim_{s \to 0+}T(s) = 0.
    \end{equation}
\end{itemize}

Such a process $T$ is also called 
\begin{itemize}
    \item a {\em process of class $L$} \cite{sato1991self}, as the distribution of $T(1)$ is of class $L$, a subclass of infinitely divisible laws studied by \Le \cite{levy1937theorie}, or 
    \item or a {\em Sato process} \cite{carr2007self}, and these processes were studied in depth by Sato \cite{sato1991self,sato1999levy}.
\end{itemize}

Here we focus on $range(T)$, the set of all values ever visited by $T$, regarded as a random subset of $\R_+$, for a non-negative, hence non-decreasing SSA process $T$ with no drift component. 
As recalled in Section \ref{sec:SSA}, it is known that for such an SSA process $T$, the \LK representation of the Laplace transform of $T(1)$ has the special form
\begin{equation}\label{eqn:lknonneg}
    \log \E e^{-u T(1)} = \int_0^\infty \left(e^{-ux} - 1\right) \frac {k(x)} x dx,
\end{equation}
for a uniquely determined right continuous non-increasing {\em key function} $k : \R_+ \to \R_+$ subject to 
\begin{equation}\label{eqn:integr}
    \int_0^\infty (x \wedge 1) \frac {k(x)} x dx < \infty.
\end{equation}

So the \Le measure of $T(1)$ has a density relative to Lebesgue measure $dx$ of the form $k(x) / x$ for such a key function $k(x)$.

Consider now the random countable set of jump times of $T$:
\begin{equation}\label{eqn:jumptimes}
    \{ s >0 : T(s) > T(s-) \} .
\end{equation}
Let $k(0+):= \lim_{x \to 0+} k(x)$.
It follows easily from self-similarity of $T$ that
\begin{itemize}
\item either $k(0+) = \infty$ and the set of jump times is dense in $\R_+$;
\item or $k(0+) < \infty$ and the jump times are the points of a scale invariant PPP with rate $\theta = k(0+)$,
when we say the SSA process $T$ has {\em (finite) \ssarate $\theta$}.
\end{itemize}

The identification of the \ssarate $\theta = k(0+)$ is part of an explicit hold-jump construction of $T$ with finite rate, provided later in Theorem \ref{thm:ktothetaJ}. 
Taking the jump sizes into consideration, as well as the jump times, the \LI representation of jumps, discussed in Section \ref{sec:jumps}, shows that the random countable set of ordered pairs of jump times and jump sizes
\begin{equation}\label{eqn:rangejump}
    \{ (s, T(s) - T(s-)) : s > 0, T(s) > T(s-)  \}
\end{equation}
is the set of points of a scale invariant Poisson point process on $\R_+ ^2$. 
This leads to the following theorem:

\begin{theorem}\label{thm:ssarate}
    For a SSA non-decreasing process $T$, with no drift component and key function $k(x)$, if $k(0+) = \theta$ is finite, then
    \begin{itemize}
        \item[(I)] $range(T):= \{ T(s): s >0 \}$ is a scale invariant point process on $\R_+$ with rate $\theta$;
        \item[(II)] the set of jump times of $T$ is a scale invariant PPP on $\R_+$ with rate $\theta$;
        \item[(III)] the set of jump sizes of $T$ is a scale invariant PPP on $\R_+$ with rate $\theta$.
    \end{itemize}
\end{theorem}
Here, (II) and (III) are just two coordinate projections of the self-similar Poisson point process \eqref{eqn:rangejump} in the positive quadrant. 
Note that in (I) it is not asserted that $range(T)$ is Poisson point process, only that this point process is scale-invariant with mean intensity measure $\theta x^{-1} dx$.
Beyond that, we know rather little about attributes of $range(T)$ as a point process on $\R_+$, such as its higher order factorial moments or Janossy measures, except in the special case when $T(1)$ is gamma distributed, and range$(T)$ turns out to be Poissonian.

To provide a more detailed description of the range of a non-decreasing SSA process $T$ with finite rate, let the scale invariant PPP of jump times of $T$ be indexed by $\Z$ in an increasing way, say 
\begin{equation}
    0 < \cdots < S_{-1} < S_0 < S_1 < \cdots < \infty,
\end{equation}
and define for each $z \in \Z$
\begin{equation}
    T_z:= T(S_z).
\end{equation}
So the jump of $T$ at time $s = S_z$ is from $T_{z-1} = T(S_z-)$ to $T_z = T(S_z)$.
See Figure \ref{fig:gamma2} for an illustration. 
Then the random countable range of $T$ is
\begin{equation}
    range(T(s), s > 0 ) = range(T_z, z \in \Z) 
\end{equation}
so the points in the range of $T$ are also indexed by $\Z$ in increasing order
\begin{equation}\label{eqn:incorder}
    0 < \cdots < T_{-1} < T_0 < T_1 < \cdots < \infty,
\end{equation}
as in the setup in for scale-invariant Poisson spacings in Theorem \ref{thm:sips}. 

Here we are not specific about how $S_0$ is selected from the random set of jump times of $T$, since the choice is not important when considering the simple point process $\{(T_z, T_{z+1} - T_z) : z \in \Z\}$ as a random set on $\R_+^2$.
It will be more convenient to define $S_0$ either so that $S_0 \le 1 < S_1$ or so that $T_0 \le 1 < T_1$.

Thus for a scale-invariant point process constructed as the range of a non-decreasing SSA process $T$ with finite \ssarate $\theta$,
comparing the point process of spacings between points, as on the left side of \eqref{eqn:sips}, and the
points themselves on the right side of \eqref{eqn:sips}, 
\begin{itemize}
    \item the spacings between points are the jump sizes of $T$:
    \begin{equation}\label{eqn:tspacings}
        range(T_{z+1} - T_z, z \in \Z ) = range( T(s) - T({s-}) > 0 : s > 0 )
    \end{equation}
which forms a scale invariant PPP with rate $\theta$, no matter what the key function $k(x)$ of $T$ with $k(0+) = \theta$;
    \item the points themselves form the range of $T$:
    \begin{equation}\label{eqn:spacon}
        range( T_z, z \in \Z ) = range\Tsvar
    \end{equation}
which is a scale invariant point process with rate $\theta$, which might or might not be Poissonian, depending on the choice of
the key function.
\end{itemize}

\begin{figure}[ht]
	\centering
	\includegraphics[width=400px]{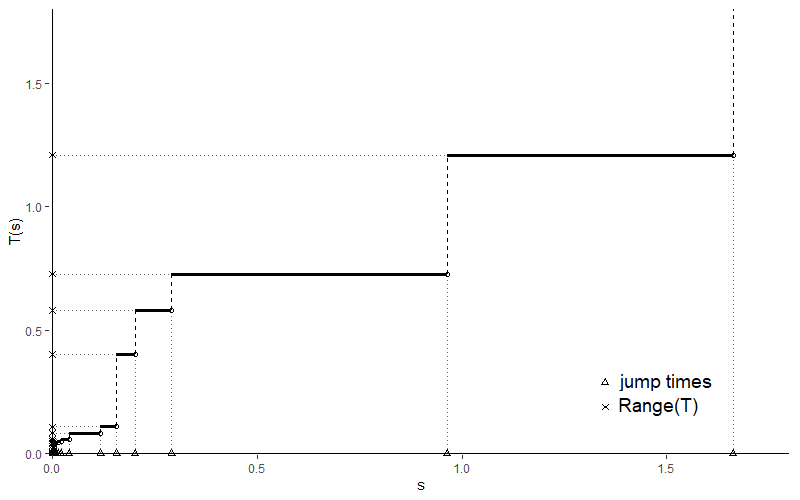}
	\caption{An SSA gamma process $\Tsvar$, its range, and jump times}\label{fig:gamma2}
\end{figure}

So to establish the spacings theorem \ref{thm:sips}, for a scale-invariant Poisson process with rate $\theta$, it only remains to show that for a suitable choice of the key function $k(x)$ with $k(0+)= \theta$, the range of $T$ is in fact Poissonian.
That scale-invariant Poisson feature of range$(T)$, with rate $\theta = 1$, was established by Dwass in the 1960s for the {\em SSA exponential process} $T$, with $k(x) = e^{-\lambda x}$ for some $\lambda >0$, so $T(1) \sim$ gamma$(1, \lambda)$ (or exp$(\lambda)$) has the exponential distribution with rate $\lambda$, with tail probability
\begin{equation}
    \P( T(1) > x ) = e^{-\lambda x } 1(x > 0).
\end{equation}

Our main point here is that the Poisson spacings theorem for a general rate $\theta >0$ is a similar consequence of the following result, which we establish in Sections \ref{sec:proof1} and \ref{sec:proofUnique}:
\begin{theorem}\label{thm:GammatoPPP}
    Fix $\theta > 0$.  For an SSA process $T := \Tsvar$, 
    \begin{itemize}
        \item range$(T)$ is a scale invariant PPP on $\R_+$ with rate $\theta$,
    \end{itemize}
    if and only if
    \begin{itemize}
        \item the distribution of $T(1)$ is gamma$(\theta, \lambda)$ for some $\lambda > 0$, with probability density
        \begin{equation}\label{gammadensity}
            \frac{d}{dx} \P( T(1) \le x ) = \frac {\lambda^\theta} {\Gamma (\theta)} x ^{\theta-1} e ^{-\lambda x} 1(x > 0).
        \end{equation}
    \end{itemize}
\end{theorem}

Implicit in this result is a well known consequence of the \LK formula:
\begin{equation}
\label{gammakey}
\mbox{ \em the gamma$(\theta,\lambda)$ distribution of $T(1)$ has key function $k(x) = \theta e^{- \lambda x }$}.
\end{equation}
However, some effort is required to pass from this fact to either of the implications of Theorem \ref{thm:GammatoPPP}.

The rest of this article is organized as follows:
\begin{itemize}
    \item Section \ref{sec:rangePPP} collects for later use some easy results about a scale invariant PPP.
    \item Section \ref{sec:SSA} recalls basic facts about SSA processes and selfdecomposable laws.
    \item Section \ref{sec:jumps} characterizes an SSA non-decreasing process $T$ with finite \ssarate by its generic jump distribution, which can be read from the key function in the \LK representation. 
    The corresponding {\em hold-jump description} of $T$ is the basis of proofs in Section \ref{sec:proof1}. 
    \item Section \ref{sec:proof1} offers proves the ``if'' part of Theorem \ref{thm:GammatoPPP}, that the range of an SSA gamma process is a PPP, in two different ways, based on the hold-jump description provided in Section \ref{sec:jumps}.
    \item Section \ref{sec:proofUnique} provides a general uniqueness theorem which includes the converse of Theorem \ref{thm:GammatoPPP}: under a technical condition, the distribution of the range of a SSA non-decreasing process uniquely determines the distribution of the process itself, up to a scale factor. 
    \item Section \ref{sec:2para} briefly discusses three different processes associated with a selfdecomposable law, then introduces a two-parameter process which provides a coupling of SSA processes with stationary independent increments in a second parameter.
    \item Section \ref{sec:history} offers some historical notes, including how the case $\theta = 1$ of Theorem \ref{thm:sips} arises in the theory of extremal processes, and how the general case $\theta > 0$ is related to the Ewens sampling formula.
\end{itemize}

\section{Scale invariant point processes}\label{sec:rangePPP}

If a point process $X:= \{X_z, z \in \Z\}$ on $\R_+$ is {\em scale invariant}, i.e.
\begin{equation} \label{eqn:scaleinv}
    range (X) \ed range( c X  ), \qquad \forall c  > 0,
\end{equation}
then $X$ is a simple point process with intensity measure $\theta x^{-1} dx$ for some $\theta$, with $0 < \theta \le \infty$.
Then we say $X$ is a scale invariant point process on $\R_+$ with {\em rate} $\theta$.
We observe that
\begin{itemize}
    \item the {\em inversion} $1/X := \{ 1/X_z, z \in \Z\}$ of a scale invariant point process $X$ on $\R_+$ is a scale invariant point process on $\R_+$ with the same rate.
\end{itemize}

By considering $L:= \log X$, this corresponds to a well known fact about stationary point processes on $\R$.
Just as not all stationary point processes on the line are reversible, not all scale invariant point process are invariant under inversion.
However, the distribution of a Poisson point process is entirely determined by its intensity measure, hence:
\begin{itemize}
    \item if $X:= \{X_z, z \in \Z\}$ is a scale invariant PPP then
    \begin{equation}\label{eqn:siptimeinv}
        range(1/X) \ed range(X).
    \end{equation}
\end{itemize}

The following lemma provides some useful characterizations of a scale invariant PPP.
Recall that for $U$ with uniform distribution on $[0,1]$, and $\theta >0$, the distribution of $U^{1/\theta}$ is the beta$(\theta,1)$ distribution on $[0,1]$ with probability density
\begin{equation}\label{betatheta}
    \frac{d}{du} \P( U^{1/\theta} \le u ) = \theta u ^{\theta -1} 1(0 < u < 1).
\end{equation}

\begin{lemma}\label{thm:sipbeta}
    Fix $x > 0$ and $\theta >0$. Suppose a scale invariant point process $X = \{ X_z, z \in \Z \}$ is indexed in increasing manner with
    \begin{equation}
        0 < \cdots < X_{-2} < X_{-1} < X_0 \le x < X_1 < X_2 < \cdots  < \infty.
    \end{equation}
Then the following three conditions are equivalent:
    \begin{itemize}
        \item $X$ is a scale invariant PPP with rate $\theta$;
        \item $x / X_1$ and $X_{z-1} / X_z$ for $z \ge 2$ are independent and identically distributed (i.i.d.) beta$(\theta,1)$ random variables;
        \item $X_0 / x$ and $X_{z-1} / X_z$ for $z < 0$ are i.i.d. beta$(\theta,1)$ random variables.
    \end{itemize}
\end{lemma}
\begin{proof}
    Consider $L := \{ L_z = \log(X_z), z \in \Z \}$. Then $L$ is a stationary PPP with intensity measure $\theta d\ell$ on $\R$ if and only if $X$ is a scale invariant PPP with rate $\theta$.
    The statements of the lemma are just transformations of well known characterizations of a stationary PPP on the line.
\end{proof}

\section{Selfdecomposable laws and SSA processes}\label{sec:SSA}

This section recalls some known results about self-similar additive processes and their one-dimensional distributions, known as the {\em selfdecomposable laws}.
Following Sato \cite{sato1991self, sato1999levy}, we call a process $T := \Tsvar$ {\em self-similar with exponent $H$} if
\begin{equation}\label{eqn:hselfsimilar}
    (T({cs}), s > 0) \ed (c^H T(s), s > 0) \qquad ( c > 0 ) ,
\end{equation}
in the sense of equality of finite-dimensional distributions. If $T$ is also additive, as defined in Section \ref{sec:intro}, we say $T$ is {\em $H$-self-similar additive ($H$-SSA)}.
We omitted the prefix `$H$-' when $H=1$ in \eqref{eqn:1selfsimilar}, because it has no impact on our discussion of $range(T)$ thanks to Lemma \ref{thm:TimeChangeEasy} below.

Easily from the definition, for an $H$-SSA process $T$, the one-dimensional distribution of $T(1)$ is {\em selfdecomposable}, meaning that for each $0 < a < 1$, there is the equality in distribution
\begin{equation}
    T(1) \ed a T(1) + R_a,
\end{equation}
for some random variable $R_a$ independent of $T(1)$. It is known \cite[Theorem 15.3]{sato1999levy} that the class of selfdecomposable laws is identical to the class $L$ introduced by \Le \cite{levy1937theorie} and Khintchine \cite{khintchine1938limit}. \Le showed that such laws are infinitely divisible with a special structure of their \Le measure. Specifically, from Sato and Yamazato \cite{sato1978distribution},
the \LK representation of a selfdecomposable distribution of $X$ is
\begin{equation}\label{eqn:lk}
    \E e^{i u X} = \exp \left\{ i b u - \hf \sigma^2 u^2 + \int_{\R - \{ 0 \}} \left(e^{iux} - 1 - \frac {i u x} {1 + u^2} \right) \frac {k(x)} x dx \right\},
\end{equation}
where $b$ is real, $\sigma^2 \ge 0$ and $k(x)$ is both
\begin{itemize}
    \item[(a)] non-negative, non-increasing on $\R_+$ and non-positive, non-increasing on $\R_-$;
    \item[(b)] subject to the usual requirement for a \Le density $k(x) / x$ that
    \begin{equation}\label{eqn:levyInt}
        \int_{\R - \{ 0 \}} (x^2 \wedge 1) \frac {k(x)} x dx < \infty.
    \end{equation}
\end{itemize}

Hence, assuming the distribution of $X$ is infinitely divisible, the distribution is selfdecomposable if and only its \Le measure has a density of the form $k(x) / x$ where $k(x)$ satisfies condition (a) above.
Sato \cite{sato1991self} gave the following uniqueness theorem for the relationship between selfdecomposable laws and $H$-SSA processes.

\begin{theorem}\label{thm:satoSD}
    For each $H$-SSA process $\Tsvar$, the marginal distribution of $T(s)$ is selfdecomposable. And for each selfdecomposable distribution $\mu$ and each $H > 0$, there exists an $H$-SSA process $\Tsvar$, unique in finite-dimensional distributions, such that $T(1)$ has distribution $\mu$.
\end{theorem}

We restrict the discussion to the case of $H = 1$ for most of the rest of this article, thanks to the following lemma.

\begin{lemma}\label{thm:TimeChangeEasy}
    Suppose $\Tsvar$ is an $H$-SSA non-decreasing process. Then the time change
    \begin{equation}
        \Tilde{T}(s) = T({s ^ {1 / H }}) \qquad (s > 0),
    \end{equation}
    gives a $1$-SSA non-decreasing process $(\Tilde{T}(s), s \ge 0)$ with $range(\tilde T) = range(T)$.
\end{lemma}

We are primarily interested in the case of a $1$-SSA process $T$ that is non-decreasing and with no drift.
Then \eqref{eqn:lk} and \eqref{eqn:levyInt} reduce to the formulas \eqref{eqn:lknonneg} and \eqref{eqn:integr} for the Laplace transform of $T(1)$.

So the \Le density of $T(s)$ at $x >0$ is $k(x / s) / x$. 
As a result, the distribution of $T$ is fully characterized by the {\em key function} $k(x)$, as detailed further in Section \ref{sec:jumps}.

\section{Hold-jump description}\label{sec:jumps}

This section presents the relationship between the \ssarate $\theta < \infty$ and the key function $k(x)$
of a $1$-SSA non-decreasing process with finite \ssarate $\theta$. 
The hold-jump description after the jump over $1$ will then be introduced as a framework to describe $range(T) \cap [1, \infty)$, as required to check whether $range(T)$ is Poisson using Lemma \ref{thm:sipbeta}. 

\subsection{Rate and generic jump} 

\begin{theorem}\label{thm:ktothetaJ}
    Suppose $T := \Tsvar$ is a $1$-SSA non-decreasing process with no drift, finite \ssarate $\theta$ and key function $k(x)$. Then $k(0+) = \theta$ and
    \begin{equation}\label{eqn:sumjump}
        T(s) = \sum_{z \in \Z} S_z J_z 1(S_z \le s), \qquad (s > 0)
    \end{equation}
    where
    \begin{itemize}
        \item the jump times $(S_z, z \in \Z )$ are the points of a scale invariant PPP with rate $\theta$;
    \end{itemize}
    Assuming also that these jump times are listed in an order depending only on their point process \eqref{eqn:jumptimes}, for instance in an increasing order with $S_0$ the time of the first jump after time $s = 1$,
    \begin{itemize}
        \item the {\em normalized jumps} $(J_z, z \in \Z)$ form a sequence of i.i.d. copies of a positive random variable $J$,
        called the {\em generic jump} of $T$, with tail probability
        \begin{equation}\label{eqn:tailprob}
            \P(J > x) = \frac {k(x)} {\theta}, \qquad (x \ge 0) ;
        \end{equation}
    \item the sequence of jump times $(S_z, z \in \Z )$ is independent of the sequence of normalized jumps $(J_z, z \in \Z )$.
    \end{itemize}
\end{theorem}

\begin{proof}
    The self-similarity of $T$ ensures that there are no jumps at fixed times.
    Thus, as a non-decreasing additive process, $T(s)$ has the following almost surely unique \LI representation \cite[Theorem 16.3]{kallenberg2021foundations}:
    \begin{equation}\label{eqn:SSAPPPrep}
        T(s) = a_s + \int_0^s \int_0^\infty x \eta(dy~dx), \qquad ( s > 0 ).
    \end{equation}
    where $a_s$ is a non-decreasing function with $a_0 = 0$ and $\eta(dy~dx)$ is a Poisson point process on $(0, \infty) ^2$ satisfying
    \begin{equation}\label{eqn:upppInt}
        \int_0^s \int_0^\infty (x \wedge 1)\E \eta(dy~dx) < \infty. \qquad ( s > 0 )
    \end{equation}
    In fact, $\eta$ characterizes the jump structure of $T$
    \begin{equation}
        \eta(\cdot) = \sum_y 1\{ (s, \Delta_y) \in \cdot \}
    \end{equation}
    where the summation extends over all times $y>0$ with $\Delta_y := T(y) - T({y-}) > 0$.
    Hence 
    \begin{equation} \label{eqn:SSAJumprep}
        T(s) = a_s + \sum_{y} \Delta_y 1(y \le s) \qquad ( s > 0 ).
    \end{equation}
    We call this $\eta$ {\em the underlying Poisson point process} of the non-decreasing SSA process $T$.
Taking into account self-similarity of $T$, it is easy to see that $a_s \equiv a_1 s$. If we further require $T$ to be a pure jump process, then $a_s \equiv 0$. 
Moreover, as we know the jump times are from a scale invariant PPP with rate $\theta$, the representation \eqref{eqn:sumjump} follows immediately from self-similarity. To finish the proof, it remains to show \eqref{eqn:tailprob}.
    
    Thanks again to self-similarity, note that $\{ (S_z, \Delta_{S_z} = S_z J_z), z \in \Z \}$ is an exhaustive listing of points from the underlying Poisson point process $\eta$ on $\R_+^2$ with intensity measure
    \begin{equation}\label{eqn:underPPPintensity}
        \nu(ds ~ dy) := \E \eta(ds ~ dy) = \theta s^{-1} ds ~ \P(sJ \in dy).
    \end{equation}
    
    Hence, for each Borel set $B \subset \R_+$ and $a > 0$,
    \begin{align} 
        \int_0^a \theta s^{-1} \P(sJ \in B) ds &= \int_{y \in \R_+} d F_J(y)\int_{s \in \R_+} \theta s^{-1} 1(sy \in B, s < a) ds\\
        (\text{set } s = x / y) \qquad &= \int_{y \in \R_+} d F_J(y)\int_{x \in \R_+} \theta x^{-1} 1(x \in B, y > x / a) dx \\
        &= \int_{x \in B} \theta x^{-1} dx \int_{y > x/a} d F_J(y)  \\
        \label{eqn:levyMeasureTransform} &= \int_{x \in B} \frac{\theta ~ \P(J > x / a)} x dx.
    \end{align}
    
    But the \Le measure of $T(1)$ is
    \begin{equation}
        \frac{k(x)} x dx = \int_{s \in (0, 1]} \theta s^{-1} \P(sJ \in dx) ds,
    \end{equation}
    which implies \eqref{eqn:tailprob} by setting $a = 1$ in \eqref{eqn:levyMeasureTransform}.

\end{proof}

To illustrate Theorem \ref{thm:ktothetaJ}, we give two examples.

\begin{example}\label{eg:gammakey}
    {\em Fix $\theta, \lambda > 0$. A $1$-SSA gamma process $\Tsvar$ with $T(1) \sim$ gamma$(\theta, \lambda)$ has exponential generic jump $J \sim$ exp$(\lambda)$ and finite \ssarate $\theta$, thanks to the key function given in \eqref{gammakey}.}
\end{example}

\begin{example}
    {\em This example is implicit in the discussion of the least concave majorant of a one-dimensional Brownian motion by Groeneboom \cite{groeneboom1983concave} and Pitman and Ross \cite{pitman2012greatest}.
    Following the notations in \cite{groeneboom1983concave}, let $\omega$ denote a standard Brownian motion starting at the origin.
    For $a>0$, let $\sigma(a)$ be the last time that maximum of $\omega(t) - at$ is attained, and $\tau(a) := \sigma(1/a), a > 0$ and $\tau(0) = 0$.
    The process $\tau$ is $2$-SSA non-decreasing process with no drift.
    By Lemma \ref{thm:TimeChangeEasy} it can be reduced to fit our $1$-SSA framework through the time change
    \begin{equation}
        T(a) := \tau(\sqrt a) = \sigma(1 / \sqrt a).
    \end{equation}
    
    Then $T$ is a $1$-SSA process, whose range is the random set of times of vertices of the least concave majorant of Brownian motion.
    The underlying Poisson point process driving $T$ has intensity measure $\nu$ which can be read from Groeneboom's description of the process $\tau$:
    \begin{equation} \label{eqn:Poisson_density_concave_majorant}
        \nu(ds~dx) = \frac 1 {s \sqrt{x}} \phi \left(\sqrt{\frac {x} {s}} \right) \frac 12 s^{-\hf} ds dx = |k'| \left(\frac xs \right) \frac {ds \, dx} {s^2},
    \end{equation}
    with $\phi(\cdot)$ the standard normal probability density function, and
    \begin{equation}
    |k'|(x) = - \frac{d k(x)}{x} = \frac 12 \cdot  \frac 1 {\sqrt{2\pi}} x^{-\hf} e^{-\hf x} \qquad (x > 0),
    \end{equation}
    the absolute first derivative of the key function $k(x) = \P( J > x) /2$
    for $J \sim$ gamma($\hf, \hf$), indicating that $T$ has \ssarate $\hf$. Therefore, the known results that
    \begin{itemize}
        \item $range(T)$ is a scale invariant point process on $\R_+$ with rate $\hf$ \\
            \cite[Corollary 9]{pitman2012greatest}, and
        \item the set of jump times of $T$ is a scale invariant PPP on $\R_+$ with rate $\hf$ \\
            \cite[Theorem 2.1]{groeneboom1983concave}
    \end{itemize}
    are the instances of parts I and II of Theorem \ref{thm:ssarate} for this particular $1$-SSA process $T$.
    Because the key function $k(x)$ of $T(1)$ is not just an exponential,
    the ``only if'' part of Theorem \ref{thm:GammatoPPP} shows that range $(T)$,
    the set of times of vertices of the least concave majorant of Brownian motion, is
    not a Poisson point process.  See Pitman and Ouaki \cite{ouaki2021markovian} for a deeper study of Markovian structure in the concave majorant of Brownian motion.}
\end{example}

Theorem \ref{thm:ktothetaJ} also yields the following {\em hold-jump description} of a $1$-SSA non-decreasing process. 

\begin{corollary}\label{thm:HoldJumpShort}
For each fixed time $s>0$, and $t \ge 0$, the future of $T$ after time $s$, conditional on $T(s) = t$, can be constructed by
\begin{itemize}
    \item {\em `Hold'} - at level $t$ till the random time $H_s \ed s \beta ^{-1}$ for $\beta \sim$ beta$(\theta, 1)$, i.e.
    \begin{equation}\label{eqn:hold}
        \frac{d}{dx} \P(H_s \le x) = \theta x^{-\theta - 1} s^{\theta} 1(x > s);
    \end{equation}
    \item {\em `Jump'} - up by $H_s J$, where $J$ is the generic jump and is independent of $H_s$, i.e.
    \begin{equation}\label{eqn:jump}
        \P(T(H_s) - T(H_s-) > y \mid H_s = s') = \P(s' J > Y) = \P(J > Y / s');
    \end{equation}
    \item then repeat, conditioning on $T(s') = t'$ for $s' = H_s, t' = t + H_s J$.
\end{itemize}

By setting a fixed starting time $S_1 = s$, \eqref{eqn:hold} and \eqref{eqn:jump} specify a homogeneous pure jump-type Markov process $((S_n, T_n), n \in \Z_+)$ with state space $\R_+^2$, whose {\em entrance law} is given by $T_1 = T(s) \ed s T(1)$.
Moreover, apart from this entrance law, the same description applies with the fixed time $s$ replaced by any stopping time $\sigma$ relative to the filtration of $T$, on the event $(\sigma > 0 )$.
\end{corollary}

The following corollary of Theorem \ref{thm:ktothetaJ} gives all possible distributions of generic jumps.

\begin{corollary}\label{thm:logfinite}
    For each $1$-SSA non-decreasing process $T$ with finite \ssarate $\theta$, its generic jump $J$ satisfies
    \begin{equation}
\label{eqn:logInt}
        \E \log^+ (J) < \infty.
    \end{equation}
    Conversely, for each $\theta < \infty$ and each positive random variable $J$ satisfying \eqref{eqn:logInt}, there exists a unique $1$-SSA process $T$ with no drift, \ssarate $\theta$ and generic jump $J$.
\end{corollary}
\begin{proof} 
    As per \eqref{eqn:tailprob}, each right continuous, non-decreasing function $k(x)$ uniquely determines $\theta$ and $J$, and vice versa. So the only thing to check is the integrable condition of the \Le density $k(x)/x$ in \eqref{eqn:integr},
    \begin{equation}
        \int_0^\infty (x \wedge 1) \frac {k(x)} x dx = \theta \left[ \int_0^1 \P(J > x) dx + \int_1^\infty \P(J > x) x^{-1} dx \right] < \infty,
    \end{equation}
    where the integral is restricted on $\R_+$ since $T$ is non-decreasing. The first term is finite no matter what distribution of $J$ is, while
\begin{align}
        \int_1^\infty \P(J > x) x^{-1} dx &= \int_1^\infty \P(\log J > \log x) d (\log x)\\
        &= \int_0^\infty \P(\log J > r) dr = \E \log^+ J,
\end{align}
    which finishes the proof.
\end{proof}

The convergence condition \eqref{eqn:logInt} appeared first in Vervaat \cite[Theorem 1.6b]{vervaat1979stochastic}, in the discussion of stochastic difference equations. 
See also Wolfe \cite[Theorem 1]{wolfe1982continuous}.

\subsection{The jump over 1}\label{subsec:jumpover1}

In considering whether or not $range(T)$ is a Poisson process, Lemma \ref{thm:sipbeta} shows it is sufficient to examine the ratios of adjacent points in $range(T) \cap (1, \infty)$ only.
The hold-jump description provides us with sufficient information to calculate the ratios as long as we know the joint distribution of where 
and when the jump of $T$ over $1$ is made. That is given by the following lemma.

\begin{lemma}\label{thm:Jumpover1lemma}
    Consider the jump over level $t$ of a 1-SSA non-decreasing process $\Tsvar$ with no drift and finite \ssarate $\theta$.
    Suppose the jump is made at time $S(t)$ from $G_t := T(S(t-)) \le t$ to $D_t := T(S(t)) > t$.
    Then
    \begin{equation}\label{eqn:jumpover1density}
        \P (S(t) \in ds, G_t \in dg, D_t - G_t > y) = \frac {\theta ds} s \P (T(s) \in dg) \P (sJ > y), \quad \forall 0 \le g \le t \le g+y.
    \end{equation}
    
    In particular, for $t = 1$ and $y = 1-g$,
    \begin{equation}\label{eqn:s1g1density}
        \P (S(1) \in ds, G_1 \in dg) = \frac {\theta ds} s \P(T(s) \in dg) \P(sJ > 1-g), \quad \forall 0 \le g \le 1.
    \end{equation}
\end{lemma}

\begin{proof} 
Since $T$ has independent increments and the set of jump times is not dense,
    \begin{align}
        \P (S(t) \in ds, G_t \in dg &, D_t - G_t > y) = \P(T(s-) \in dg) \P( T(s + ds) - T(s-) > y) \notag\\
        &= \P(T(s-) \in dg) \int_{x \in (y, \infty)} \nu (ds, dx) \notag\\
        &=  \P(T(s) \in dg) \int_{x \in (y, \infty)} \frac {\theta ds} {s} F_J \left( \frac {dx} s \right) \notag\\
        &= \frac {\theta ds} s \P (T(s) \in dg) \P (sJ > y), \quad \forall 0 < g \le t \le g+y.
    \end{align}
\end{proof}

Now we give the following theorem as the {\em hold-jump description after the jump over $1$}.

\begin{theorem}\label{thm:HoldJump}
    Suppose $T := \Tsvar$ is a $1$-SSA non-decreasing process with no drift and finite \ssarate $\theta$. Let $J$ denote its generic jump. Let $S_1$ be the time when $T$ jumps over $1$
    \begin{equation}
        T(S_1-) \le 1 < T(S_1),
    \end{equation}
    and $S_1 < S_2 < \cdots$ be the times of successive jumps $s$ of $T$ with $T(s)> T(s-) $ and $T(s) > 1$, and define
    \begin{equation}
        T_0 := T(S_1-) \le 1 < T_1 := T(S_1) < T_2:= T(S_2) < \cdots
    \end{equation}
    so that $T_1, T_2, \ldots$ is the increasing sequence of values greater than $1$ which are attained by $T$ on the successive intervals $[S_1, S_2), [S_2, S_3), \ldots$.
Then the joint distribution of the two sequences 
    $(S_n, n \ge 1)$ and  $(T_n, n \ge 0)$ is determined as follows:
    \begin{align}
        \label{eqn:s1t0}\P(S_1 \in ds, T_0 \in dt) & = \frac {\theta ds} s \P(T(s) \in dt) \P(sJ > 1-t),\\
        \label{eqn:srec} S_n & = S_1 \left(\prod_{i=1}^{n-1} \beta_i \right) ^ {-1} \qquad ( n = 2,3, \ldots ),\\
        \label{eqn:trec} T_n & =  T_{n-1} + S_n J_n \qquad ( n = 1,2, \ldots ),
    \end{align}
    where $\beta_1, \beta_2, \ldots$ and $J_1, J_2, \ldots$ are
    independent random variables, with 
    \begin{itemize}
    \item $\beta_1, \beta_2, \ldots$ all with the beta$(\theta,1)$ distribution \eqref{betatheta};
    \item $J_1, J_2, \ldots$ identically distributed copies of the generic jump $J$.
    \end{itemize}
\end{theorem}

\begin{proof}
    Formula \eqref{eqn:s1t0} is just \eqref{eqn:s1g1density} with different notation, while \eqref{eqn:srec} and \eqref{eqn:trec} are due to the hold-jump description presented in Corollary \ref{thm:HoldJumpShort}.
\end{proof}

\subsection{Proof of Theorem \ref{thm:ssarate}}

Suppose that $T$ is a non-decreasing $1$-SSA process with no drift and finite \ssarate $\theta$.
Then Part (II) holds by definition, while (III) is observed by setting $a = \infty$ in \eqref{eqn:levyMeasureTransform}.
To see (I), set $g = t$ and $y = 0$ in \eqref{eqn:jumpover1density} of Lemma \ref{thm:Jumpover1lemma}.
Suppose the rate of $range(T)$ is $\alpha$, then for each Borel set $B \subset \R_+$,
    \begin{align}
		\int_B \frac {\alpha dt} t &= \int_{s \in (0, \infty)} \int_{t \in B} \P(S(1) \in ds, G_1 \in dt, D_1 - G_1 > 0) \notag\\
		&= \int_{s \in (0, \infty)} \frac {\theta ds} s \int_{t \in B}  \P(T(s) \in dt) \P(sJ > 0) \notag\\
		&= \int_{s \in (0, \infty)} \frac {\theta ds} s \int_{t \in s^{-1}B} \P(T(1) \in dt) \notag\\
		&= \int_{s \in (0, \infty)} \int_{t \in s^{-1}B} \frac {\theta ds} s d F_{T(1)}(t) \notag\\
		&= \int_{t \in (0, \infty)} \int_{s \in t^{-1}B} \frac {\theta ds} s d F_{T(1)}(t) \notag\\
		(\text{set } r = st) &= \int_{t \in (0, \infty)} \int_{r \in B} \frac {\theta dr} r d F_{T(1)}(t) = \int_{r \in B} \frac {\theta dr} r,
	\end{align}
which shows $\alpha = \theta$.

\section{Proof that the range of an SSA gamma process is a PPP}
\label{sec:proof1}

This is the ``if'' part of Theorem \ref{thm:GammatoPPP}.
We offer two different proofs.

\subsection{First proof - through the hold-jump description after the jump over 1}

This argument shows that the description of $range(T) \cap (1, \infty)$ implied by Theorem \ref{thm:HoldJump} matches that required by Lemma \ref{thm:sipbeta}. 
We follow the notation of Theorem \ref{thm:HoldJump}, which contains a complete description of the joint distribution of the first arrival times and levels $(S_n, T_n)$ at all levels $t >1$ that are ever attained by the path of $T$.
By scaling, it is enough to consider the case $\lambda = 1$.
From Example \ref{eg:gammakey}, the generic jump is exp$(1)$ distributed: $\P(J >x) = e^{-x}$ for $x > 0$, so \eqref{eqn:s1t0} becomes
\begin{equation}
    \P(S_1 \in ds, T_0 \in dt) = \frac {\theta ds} s \frac 1{s^\theta \Gamma(\theta)} t^{\theta -1} e^{-\frac ts} dt ~ e ^ {-(\frac {1-t} s)} = \frac {s^{-1-\theta} e ^{- \frac 1s} ds}{\Gamma(\theta)} \theta t^{\theta - 1} dt.
\end{equation}
That means $T_0$ and $S_1$ are independent, with $T_0$ distributed beta$(\theta, 1)$ and $S_1$ distributed as inverse gamma$(\theta, 1)$, meaning $S_1^{-1}$ distributed as the gamma$(\theta, 1)$.

Hence by the memoryless property of exponential random variables, we may assume $T_0 = 1$ without hurting the joint distribution of $(S_n, n \ge 1)$ and $(T_n, n \ge 1)$. The joint distribution is now re-written as follows:
\begin{align}
    \label{gammasrec}S_n & = \left(\gamma_\theta \prod_{i=1}^{n-1} \beta_i \right) ^ {-1} \qquad ( n = 1,2, \ldots ) \\
    \label{gammatrec}T_n & =  T_{n-1} + S_n \eps_n \qquad ( n = 1,2, \ldots ) 
\end{align}
where $\gamma_\theta, \beta_1, \beta_2, \ldots$ and $\eps_1, \eps_2, \ldots$ are independent random variables, with 
\begin{itemize}
    \item $\gamma_\theta$ assigned the gamma$(\theta, 1)$ distribution;
    \item $\beta_1, \beta_2, \ldots$ all with the beta$(\theta,1)$ distribution;
    \item $\eps_1, \eps_2, \ldots$ all with the exp$(1)$ distribution.
\end{itemize}

All the ingredients needed for calculating $range(T) \cap (1, \infty)$ are in view. But the description of the levels $T_n$ is tangled up with the description of the times $S_n$ in such a way that it is not immediately obvious why the sequence of ratios $T_{n-1}/T_n$ is also a sequence of i.i.d. copies of $beta(\theta, 1)$.
However, the argument is completed by the following lemma.
 
\begin{lemma}
    Suppose that random variables $S_1:= 1/\gamma_\theta$ and $S_{n+1}:= S_n/\beta_n$ for $n \ge 1$ are defined by the recursion \eqref{gammasrec}, along with $T_0:= 1 < T_1 < T_2 < \cdots$ by \eqref{gammatrec}, from independent random variables $\gamma_\theta$, $\beta_i$ and $\eps_i$ as above.
    Then for each $n = 1,2, \ldots$, the $n+1$ ratios
    \begin{equation}\label{nrats}
        \frac{T_0}{T_1}, \ldots, \frac{T_{n-1}}{T_n}, \frac{T_n}{S_n}
    \end{equation}
    are independent, with the first $n$ consecutive $T$-ratios all distributed according to the common beta$(\theta,1)$ distribution of all the $\beta_i$, and with the last of the $n+1$ ratios
    \begin{equation}\label{tsrat}
        \frac{T_n}{S_n} \ed \gamma_{\theta +1},
    \end{equation}
    the gamma$(\theta+1, 1)$ distribution.
\end{lemma}

\begin{proof}

For $n = 1$, with $T_0:= 1$
\begin{equation}
    T_1 = 1  + \frac{ \eps_1 }{\gamma_\theta}  = \frac{ \gamma_\theta + \eps_1 } { \gamma_\theta}
\end{equation}
and hence
\begin{equation}
    \frac{T_0}{T_1} = \frac{ 1 } { T_1 } = \frac{ \gamma_\theta } { \gamma_\theta + \eps_1 } \ed \beta_{\theta,1 }
\end{equation}
and this variable $T_0/T_1$ is independent of
\begin{equation}
    \frac{T_1}{S_1} := T_1 \gamma_\theta  = \gamma_\theta + \eps_1  \ed \gamma_{\theta + 1}
\end{equation}
by the beta-gamma algebra mentioned in Lukacs \cite{lukacs1955characterization}.
The case of general $n = 1, 2, 3, \ldots$ now follows by induction on $n$, starting from this base case $n = 1$.
Multiply the recursion \eqref{gammatrec} by $1/T_{n+1} = \beta_n/T_n$ to see that
\begin{equation}\label{tnsnrat}
    \frac{T_{n+1} }{S_{n+1} } = \frac{ T_n}{S_n} \, \beta_n + \eps_{n+1} \ed \gamma_\theta + \eps_1 \ed \gamma_{\theta + 1}
\end{equation}
because the gamma$(\theta + 1)$ distribution of $T_n/S_n$ and the independence of this variable and $\beta_n \ed \beta_{\theta,1}$ makes their product $(T_n/S_n) \beta_n \ed \gamma_\theta$.
Moreover
\begin{equation}
    \frac{T_{n} } {T_{n+1} } = \frac{ T_n /S_{n+1} }  { T_{n+1}/S_{n+1} }  =  \frac{ (T_n/S_n) \beta_n }  {  (T_n/S_n) \beta_n + \eps_{n+1} } \ed \frac{ \gamma_\theta } {\gamma_\theta + \eps_1 } \ed \beta_{\theta,1}
\end{equation}
and this ratio $T_n/T_{n+1}$ is independent of the ratio $T_{n+1}/S_{n+1}$ in \eqref{tnsnrat}, again by beta-gamma algebra.
Thus $T_{n}/T_{n+1}$ and $T_{n+1}/S_{n+1}$ are independent with the required distributions.
By inductive assumption, $T_n/S_n$ is independent of the $n$ ratios $T_{i-1}/T_i$ for $1 \le i \le n$.
The two variables $T_{n}/T_{n+1}$ and $T_{n+1}/S_{n+1}$ are functions of $T_n/S_n$ and two further independent variables $\beta_n$ and $\eps_{n+1}$.
Hence the required independence of the $n+2$ variables involved in \eqref{nrats} with $n+1$ in place of $n$.
\end{proof}

\subsection{Second proof - to exploit symmetry}

To see the symmetry, consider the {\em time-inversion} $\tilde T := (\tilde T(s) := T(s^{-1}), s > 0)$ of $T$.
It is obvious that $\tilde T$ has a non-increasing staircase path, which is fully determined by its {\em corners}, i.e. points on the left end of each flat of the path.

To describe the corners, we restate, in a time-reversed manner, the {\em backward hold-jump description} of $\tilde T$. Conditional on $\tilde T(s) = t$, the {\em past} of $\tilde T$ before time $s$ can be fully constructed by
\begin{itemize}
    \item {\em `Hold'} - at level $t$ till the random time $H_s \ed s \beta$ (this is going backward in time), where $\beta$ has the common $beta(\theta, 1)$ distribution;
    \item {\em `Jump'} - up by $H_s^{-1} J$, where $J$ is the generic jump and is independent of $H_s$;
    \item then repeat, conditioning on $\tilde T(s') = t'$ for $s' = H_s, t' = t + H_s^{-1} J$.
\end{itemize}

\begin{lemma}\label{thm:symmetry}
    Suppose $T$ is a $1$-SSA gamma process with $T(1) \sim$ gamma$(\theta, 1)$ and $\tilde T := (\tilde T(s) := T(s^{-1}), s > 0)$ is the time-inversion of $T$. Then the set $\CC$ of corners of the path of $\tilde T$ is symmetric about the bisectrix.
\end{lemma}

Theorem \ref{thm:GammatoPPP} then follows since
\begin{equation}
    range(T) = range(\tilde T) = \pi_t (\CC) \ed \pi_s (\CC),
\end{equation}
where $\pi_t$ and $\pi_s$ are projections onto the $t$- and $s$-axes, and $\pi_s (\CC)$ is a scale invariant PPP on $\R_+$ with rate $\theta$ thanks the invariance under inversion \eqref{eqn:siptimeinv}.

\begin{proof} (Lemma \ref{thm:symmetry})

    Consider the joint density $p_m(a_1, \cdots, a_m)$ of the event that there are $N$ consecutive points $a_i = (s_i, t_i) \in \CC$ for $i = 1, 2, \cdots, N$ indexed decreasingly in $s$-coordinate
    \begin{equation}
        s_1 > s_2 > \cdots > s_N,
    \end{equation}
    hence is indexed increasingly in $t$-coordinate. Knowing $J \sim$ gamma$(1, 1)$,
    \begin{align}
        p_N (a_1, \cdots, a_N) &= \frac {s_1^{\theta}} {\Gamma(\theta)} (t_{0})^{\theta-1} e^{-s_1 t_1 } \cdot \frac \theta {s_1} \cdot \left[ \prod_{n=1}^{N-1} \theta s_{i+1}^{\theta - 1} s_i^{-\theta} \right ] \cdot \left[ \prod_{n=1}^{N-1} s_i e^{-s_i(t_{i+1} - t_i)} \right ]\\
        \label{eqn:jointdensitytheta}&= \frac {\theta^N} {\Gamma(\theta)} (t_1 s_N)^{\theta-1} \exp \left\{ \sum_{n = 2}^{N-1} s_n t_n - \sum_{n=1}^{N-1} s_n t_{n+1} \right\}.
    \end{align}
    Everything in the first line is self-explanatory from the backward hold-jump description given above, except that the second term $\theta / s_1$ accounts for a `hold' with $\beta = 1$, meaning that there is an immediate jump at time $s_1$.
    
    Expression \eqref{eqn:jointdensitytheta} is invariant under the substitution
    \begin{equation}
        (\tilde s_1, \tilde s_2, \cdots, \tilde s_N) \leftrightarrow (\tilde t_N, \tilde t_{N-1}, \cdots, \tilde t_1),
    \end{equation}
    which proved the symmetry as desired.
\end{proof}

This proof is inspired by Gnedin \cite[Equation (5)]{gnedin2008corners} where a similar symmetry was shown for a different setup of corners constructed from a PPP on $\R_+^2$ with unit intensity.
We also remark that the the proof of the Poisson spacing theorem by Arratia, Barbour and Tavar\'e \cite[Lemma 7.1]{arratia2006tale} is also done by checking the density of consecutive points. 
However, we manage to avoid the brutal integration in their proof by exploiting symmetry.

\section{Uniqueness of the distribution of the range} \label{sec:proofUnique}

In this section, we establish the following uniqueness theorem, from which the 
``if'' part of Theorem \ref{thm:GammatoPPP} follows immediately.

\begin{theorem}\label{thm:unique}
    Suppose $T$ and $\tilde T$ are two $1$-SSA non-decreasing processes with no drift and the same finite \ssarate $\theta$. Then
    \begin{equation}
        range(T) \ed range(\tilde T) \text{\qquad implies \qquad } T \ed c \tilde T \text{\qquad for some\qquad} c > 0.
    \end{equation}
\end{theorem}

\begin{proof}
    Following (the results and also the notation of) Theorem \ref{thm:HoldJump}, we may write $S$ in terms of $T$ and $J$ in \eqref{eqn:srec}, then \eqref{eqn:trec} becomes
	\begin{equation}\label{eqn:TJ_relation}
		\dfrac {T_{n} - T_{n-1}} {T_{n+1} - T_{n}} = \dfrac {\beta_n J_{n-1}} {J _n},\qquad \forall n \ge 2,
	\end{equation}
	where (still, as in Theorem \ref{thm:HoldJump}) $1 < T_1 < T_2 < \cdots$ is an exhaustive ordered listing of points of $T$ on $(1,\infty)$, $\beta_n$ are i.i.d. Beta($\theta, 1$)'s and $J_n$ are i.i.d. copies of the generic jump $J$.
	
	Similarly, for another $1$-SSA process $\tilde T$ with \ssarate $\theta$ whose range is equal in distribution as the one of $T$, we have
	\begin{equation}\label{eqn:TJ_relation_tilde}
	    \dfrac {\tilde T_{n} - \tilde T_{n-1}} {\tilde T_{n+1} - \tilde T_{n}} = \dfrac {\tilde \beta_n \tilde J_{n-1}} {\tilde J _n},\qquad \forall n \ge 2,
	\end{equation}
	where every random variable is similarly defined as in \eqref{eqn:TJ_relation} for $\tilde T$ instead of $T$.
	
	By assumption, $T$ and $\tilde T$ are both $1$-SSA with the same finite \ssarate $\theta$. So it is sufficient to show $J \ed \tilde J$.
	
	Note that $range(T) \ed range(\tilde T)$ implies
	\begin{equation}
		\left(T_{n}, n \ge 1 \right) \ed \left(\tilde T_{n}, n \ge 1 \right).
	\end{equation}

	Therefore, by \eqref{eqn:TJ_relation} and \eqref{eqn:TJ_relation_tilde},
	\begin{equation}\label{bar_nobar_with_beta}
		\left(\dfrac {\beta_n J_{n-1}} {J_n}, n \ge 2 \right) \ed \left(\dfrac {\tilde \beta_n \tilde J_{n-1}} {\tilde J_n}, n \ge 2 \right).
	\end{equation}
	
	By Lemma \ref{thm:MultiSimpRV} below, \eqref{bar_nobar_with_beta} can be simplified to
	\begin{equation}\label{bar_nobar}
		\left(\dfrac {J_{n-1}} {J_n}, n \ge 2 \right) \ed \left(\dfrac {\tilde J_{n-1}} {\tilde J_n}, n \ge 2 \right)
	\end{equation}
	which implies $J \ed c \tilde J$ thanks to Lemma \ref{thm:diffEd}.
\end{proof}

This Lemma \ref{thm:MultiSimpRV} is a multivariate extension of a simplified version of exercise 1.13.1 in Chaumont and Yor \cite{chaumont2012exercises}, since we only need the case when all coordinates are strictly positive.

Recall $\R_+ := (0, \infty)$. Suppose $Y = (Y_1, Y_2, \cdots, Y_n)$ is a $\R_+^n$-valued random variable. Let $\Phi_{\log(Y)}$ denote the characteristic function of $\log Y := (\log Y_1, \log Y_2, \cdots, \log Y_n)$
\begin{equation}
    \Phi_{\log(Y)}(\lambda) := \E \exp \{ i\lambda \cdot \log(Y) \}, \qquad \lambda \in \R^n,
\end{equation}
where $\cdot$ is the usual inner product of vectors.

\begin{lemma}[Multivariate simplifiable random variables]\label{thm:MultiSimpRV}
	If the non-zero set of the characteristic function of $(\log Y)$, i.e. $\{ \lambda: \Phi_{\log Y} (\lambda) \neq 0\}$, is dense in $\R^n$, then $Y$ is {\em multivariate simplifiable}, i.e. for all $\R_+^n$-valued random variables $X, Z$ independent of $Y$,
	\begin{equation}
	    X \times Y \ed Z \times Y \text{\qquad implies \qquad } X \ed Z,
	\end{equation}
	where $\times$ denotes the entry-wise product $X \times Y := (X_1 Y_1, X_2 Y_2, \cdots, X_n Y_n)$.
	
	In particular, for each $\theta > 0$, if $Y_1, Y_2, \cdots Y_n$ are i.i.d. copies of $\beta \sim$ beta($\theta, 1$), $Y$ is multivariate simplifiable.
\end{lemma}

\begin{proof}
	\begin{multline}
	    \Phi_{\log (X\times Y)} (\lambda) = \E \exp[i\lambda \cdot \log(X \times Y)] = \E \exp[i\lambda \cdot (\log X + \log Y)]\\
	    = \Phi_{\log X} (\lambda) \Phi_{\log Y} (\lambda), \qquad \forall \lambda \in \R^n.
	\end{multline}
	and similarly $\Phi_{\log (Z\times Y)} (\lambda) = \Phi_{\log Z} (\lambda) \Phi_{\log Y} (\lambda)$. Hence, $X \times Y \ed Z \times Y$ implies
	\begin{equation}
	    \Phi_{\log X} (\lambda) \Phi_{\log Y} (\lambda) = \Phi_{\log Z} (\lambda) \Phi_{\log Y} (\lambda), \qquad \lambda \in \R^n.
	\end{equation}
	
	Then the cancellation of $\Phi_{\log Y} (\lambda)$ on a dense subset $\{ \lambda: \Phi_{\log Y} (\lambda) \neq 0\}$ of $\R^n$ shows
	\begin{equation}
	    \log X \ed \log Z
	\end{equation}
	which easily implies $X \ed Z$.
	
	Now it is left to show that $\{ \lambda: \Phi_{\log Y} (\lambda) \neq 0\}$ is dense if $Y_1, Y_2, \cdots Y_n$ are i.i.d. copies of $\beta \sim$ beta($\theta, 1$). 
	
    Observe that the characteristic function of $\log \beta$ is non-zero on $\R$
    \begin{equation}
        \phi_{\log \beta}(\lambda) = \E \exp[i\lambda \log(\beta)] = \E \beta_k^{i\lambda} = \int_0^1 x^{i\lambda} ~ \theta x^{\theta - 1} dx = \frac \theta {\theta + i\lambda}, \qquad \forall \lambda \in \R.
    \end{equation}
    Hence the characteristic function of $\log Y$ is non-zero on $\R^n$.
\end{proof}

\begin{lemma}\label{thm:diffEd}
    Suppose $X$ and $Y$ are two positive random variables, with i.i.d. copies $X_n, Y_n, n = 1, 2, \cdots$, respectively. If for each $n$,
    \begin{equation}\label{eqn:diffEd}
        \left(\frac{X_1}{X_2}, \frac{X_2}{X_3}, \cdots, \frac{X_{n-1}}{X_n}\right) \ed \left(\frac{Y_1}{Y_2}, \frac{Y_2}{Y_3}, \cdots, \frac{Y_{n-1}}{Y_n}\right),
    \end{equation}
    then $Y \ed c X$ for some constant $c$.
\end{lemma}

\begin{proof}
    Identity \eqref{eqn:diffEd} is equivalent to
    \begin{equation}
        \left(\frac{X_1}{X_2}, \frac{X_1}{X_3}, \cdots, \frac{X_1}{X_n}\right) \ed \left(\frac{Y_1}{Y_2}, \frac{Y_1}{Y_3}, \cdots, \frac{Y_1}{Y_n}\right),
    \end{equation}
    which implies
    \begin{equation}\label{eqn:extraZ}
        \left(X_1 \frac{Z_2}{X_2}, X_1 \frac{Z_3}{X_3}, \cdots, X_1 \frac{Z_n}{X_n}\right) \ed \left(Y_1 \frac{Z_2}{Y_2}, Y_1 \frac{Z_3}{Y_3}, \cdots, Y_1 \frac{Z_n}{Y_n}\right),
    \end{equation}
    where $Z_n, n = 2, 3, \cdots$ are i.i.d. copies of $Z \sim$ exp$(1)$ and are independent of the $X$- and $Y$-sequences. 
    
    The extra randomization $Z$ is only to ensure that the cumulative distribution function of $Z/X$ is continuous on $\R_+$.
    When $n := 2k \to \infty$, the $(k+1)$-th order statistic of the left sequence of \eqref{eqn:extraZ} converges almost surely to $m_x X_1$ where $m_x > 0$ is the median of $Z/X$. Similarly, the $(k+1)$-th order statistic of the right sequence converges almost surely to $m_y Y_1$ where $m_y > 0$ is the median of $Z/Y$.
    
	Hence
	\begin{equation}
	    m_x X_1 \ed m_y Y_2,
	\end{equation}
	which implies
	\begin{equation}
	    X \ed c Y \text{\qquad with \qquad} c = \frac {m_y} {m_x}.
	\end{equation}
\end{proof}

\section{Processes associated with selfdecomposable laws}\label{sec:2para}

In this section, we first introduce three different kinds of one-parameter processes associated with selfdecomposable laws, which are known in the previous study related to selfdecomposable laws.
Then we introduce a two-parameter process with selfdecomposable margins that behaves differently along its two parameters.

\subsection{One-parameter processes}

Let $X$ be a selfdecomposable random variable.
According to the result of Sato \cite{sato1991self} presented in Theorem \ref{thm:satoSD}, 
\begin{itemize}
    \item for each $H > 0$, there is a unique distribution of an  $H$-SSA process $T^{(H)} := (T^{(H)}(s), s > 0)$ such that $T^{(H)}(1) \ed X$. 
\end{itemize}
On the other hand, since $X$ is infinitely divisible,
\begin{itemize}
    \item there is a unique \Le process $U := (U(s), s > 0)$ with $U(1) \ed X$.
\end{itemize}
See Sato \cite[Section 4]{sato1991self} for a comparison between these two processes, where it is mentioned that $T^{(H)} \ed U$, in the sense of equality in finite-dimensional distributions, if and only if $X$ is a constant variable $0$ or strictly stable with index $1/H$.

A third process associated with selfdecomposable $X$ is known as the {\em background driving \Le process} ({\em BDLP}) first discussed by Wolfe \cite{wolfe1982continuous} and Jurek and Vervaat \cite{jurek1983integral}, and named by Jurek \cite{jurek1997selfdecomposability}. 
A \Le process $Y := (Y(r), r > 0)$ is called the BDLP of $X$ if

\begin{equation}
    X \ed \int_0^\infty e^{-r} dY(r).
\end{equation}

The relationship of $T^{(H)}$ and $Y$ is given by

\begin{equation}\label{eqn:ssatobd}
    Y(r) = \int_{e^{-r}}^1 s^{-H} d T^{(H)}(s), \qquad r > 0,
\end{equation}
and
\begin{equation}\label{eqn:bdtossa}
    T^{(H)}(s) = \int_{-\log(s)}^\infty e^{rH} dY(r), \qquad 0 < s < 1.
\end{equation}

To illustrate the differences among these three processes, we observe if $k(x) / x$ is the \Le density of $X$, then
\begin{itemize}
    \item the \Le density of $T^{(H)}(s) \ed T^{(1)}(s^H)$ is given by $k(s^{-H} x) / x$;
    \item the \Le density of $U(s)$ is given by $s k(x) / x$;
    \item the \Le density of $Y(r)$ is given by $r k(x)$.
\end{itemize}

The BDLP, denoted $Y$, can be easily extended for $s \le 0$ by setting $Y(0) = 0$ and $(-Y(-s), s > 0)$ an independent copy of $(Y(s), s > 0)$, whence \eqref{eqn:ssatobd} holds for $r \in \R$ and \eqref{eqn:bdtossa} holds for $s > 0$. See Jeanblanc, Pitman and Yor \cite[Theorem 1]{jeanblanc2002self}.

\subsection{A two-parameter process}

For simplicity, suppose the selfdecomposable random variable $X$ is also non-negative with \Le triple $(0, 0, k(x) x^{-1} dx)$ where $\theta := k(0+) < \infty$. 
So the associated $1$-SSA process $\Tsvar$ with $T(1) \ed X$ is non-decreasing with no drift, finite \ssarate $\theta$ and generic jump $J$ whose distribution is determined by
\begin{equation}
    \P(J > x) = k(x) / \theta.
\end{equation}

Moreover, $T(s)$ has the following \LI representation
\begin{equation}
    T(s) = \int_0^s \int_0^\infty x \eta(dy ~ dx) \qquad (s > 0),
\end{equation}
where $\eta$ is a Poisson point process on $\R_+^2$ satisfying \eqref{eqn:upppInt}, with intensity measure
\begin{equation}
    \nu(dy ~ dx) := \E \eta(dy ~ dx) = \theta y^{-1} dy \P(sJ \in dx).
\end{equation}

Based on $\eta$, define a Poisson point process on $\R_+^3$, say $\eta^*$, by setting its intensity measure $\nu^* := \E \eta^* = \nu \bigotimes \lambda$, where $\lambda$ is the ordinary Lebesgue measure on $\R_+$.
Now define a two-parameter process $(T(s, w), s > 0, w > 0)$ by setting
\begin{equation}
    T(s, w) := \int_{z \in (0, w]} \int_{y \in (0, s]} \int_{x > 0} x \eta^*(dz ~ dy ~ dx) \qquad (s, w > 0).
\end{equation}

Then the following properties of this two-parameter process are easily checked:
\begin{itemize}
    \item $T(s, w)$ is non-decreasing in both $s$ and $t$;
    \item for each fixed positive integer $w$, $(T(s, w), s > 0)$ is a $1$-SSA process obtained by adding $w$ independent copies of $T$;
    \item for each fixed $w > 0$, $T_{(w)} := T(\cdot, w) = (T(s, w), s > 0)$ is a $1$-SSA process with finite \ssarate $\theta w$ and generic jump $J$; 
    \item for each fixed $s > 0$, $T^{(s)} := T(s, \cdot) = (T(s, \theta), \theta > 0)$ is a subordinator;
    \item the family $(T_{(w)}, w > 0)$ is coupled such that it is a ``subordinator of $1$-SSA processes'', i.e. for fixed $u > v > 0$, the increment $T_{(u)} - T_{(v)} \ed T_{(u - v)}$ is a $1$-SSA non-decreasing process independent of $(T(s, w), s > 0, 0 < w < v)$;
    \item the family $(T^{(s)}, s > 0)$ is coupled such that it is an ``$1$-SSA family of subordinators'', i.e. for fixed $s > 0$, $T^{(s)} \ed s T^{(1)}$ and for fixed $u > v > 0$, the increment $T^{(u)} - T^{(v)}$ is a subordinator independent of $(T(s, w), 0 < s < v, w > 0)$;
    \item for each finite interval $I$, when $T$ is not identically $0$, the jumps of $T^{(s)} := (T^{(s)}(w) = T(s, w), w > 0)$ as a subordinator are almost surely dense on $I$;
    \item however, for fixed $u > v > 0$, the number of jumps of $T^{(u)} - T^{(v)}$ is almost surely finite on $I$.
\end{itemize}

We only discussed above the case when $T$ is non-decreasing and with finite \ssarate. But it is easy to see that the bivariate process can be defined more generally. In particular, it is easy to add a deterministic or Brownian component. But the case when $T$ has infinite \ssarate and jumps of both signs would require more care.

\section{Historical remarks}\label{sec:history}

\paragraph{Selfdecomposable laws and self-similar processes}

The class of selfdecomposable laws was first studied by Levy \cite{levy1937theorie} and Khintchine \cite{khintchine1938limit}, as an extension to stable laws as the limit distributions for sums of identically distribute random variables.
Self-similar processes were first studied Lamperti \cite{lamperti1962semi} as a generalization to stable processes, which he called {\em semi-stable processes}.

\begin{theorem}[Lamperti \cite{lamperti1962semi}]
    Let $\Tsvar$ be an additive process. Suppose $\Tsvar$ is also {\em semi-stable}, that is, for every $c > 0$, there exists a constant $b(c)$ such that
    \begin{equation}
        (T({cs}), s \ge 0) \ed (b(c) T(s), s \ge 0),
    \end{equation}
    then for some $H \ge 0$
    \begin{equation}
        b(c) = c^H \qquad ( c > 0 ).
    \end{equation}
\end{theorem}

Note that if $H = 0$, the process is trivial; otherwise, it is a $H$-SSA process.

The integral representation of selfdecomposable laws was studied by Wolfe \cite{wolfe1982continuous} for $\R$-valued random variables and then by Jurek and Vervaat \cite{jurek1983integral} for more general Banach space-valued random variables.
Although they did not consider self-similar processes, the integral representation
\begin{equation}
    T(1) = \int_0^\infty e^{-s} Y(ds),
\end{equation}
for $Y(\cdot)$ the background driving \Le process (BDLP) associated with $T(1)$, is essentially the same as our integral representation \eqref{eqn:SSAPPPrep} by treating $e^{-s}$ as our index $s$.
See also Jurek \cite{jurek2019background} for a recent study of selfdecomposable laws and the associated BDLP.

Sato \cite{sato1991self, sato1999levy} investigated in detail and built the connection between the class $L$ of selfdecomposable distributions and SSA processes as stated in Theorem \ref{thm:satoSD}.
There is also a detailed background and a comprehensive list of references to earlier results on selfdecomposable laws in \cite{sato1991self}.

Jeanblanc, Pitman and Yor \cite{jeanblanc2002self} pointed out that either of the two representations by Wolfe \cite{wolfe1982continuous} and Sato \cite{sato1991self} follows easily from the other.
That reference also provides further background theory of selfdecomposable laws and their representations. 
Bertoin \cite{bertoin2002entrance} treats the entrance law of self-similar processes.
Tudor \cite{tudor2013analysis} studied the variations of self-similar processes from a stochastic calculus approach.

\paragraph{Summation representation of $T(1)$}
In the setup of Theorem \ref{thm:ktothetaJ}, the identity \eqref{eqn:sumjump} can be viewed as a decomposition of the selfdecomposable random variable $T(s)$.
For simplicity, consider only $T(1)$.
Now index $S_z$ such that
\begin{equation}
    0 < \cdots < S_{-2} < S_{-1} < S_0 \le 1 < S_1 < S_2 < \cdots < \infty.
\end{equation}

By Lemma \ref{thm:sipbeta}
\begin{equation}
    S_{1-n} = \prod_{k = 1}^n \beta_k \qquad (n \ge 1),
\end{equation}
where $\beta_k$ are i.i.d. beta($\theta, 1$) random variables.

\begin{corollary}
    Suppose $T(1)$ is a non-negative, selfdecomposable random variable. Then there exists a sequence of i.i.d. non-negative random variables $J_z$ such that
    \begin{equation}\label{eqn:sumjump2}
        T(1) \ed \sum_{n = 1}^\infty \left(\prod_{k = 1}^n \beta_k \right) J_{1-n},
    \end{equation}
    where $\beta_k$ are i.i.d. beta($\theta, 1$) random variables independent of $J_z$.
\end{corollary}

The distribution of a random variable $T(1)$ admitting the representation \eqref{eqn:sumjump2} for i.i.d. sequences $\beta_k$ (not necessarily beta) and $J_k$ was studied first by Vervaat \cite[example 3.8]{vervaat1979stochastic} as the solution of the stochastic difference equation
\begin{equation}\label{eqn:vervaatsde}
    X \ed A(X + C),
\end{equation}
with $X, A, C$ independent and $|A| < 1$. By iteration,
\begin{equation}\label{vervaatsolution}
    X \ed \sum_{n=1}^\infty \left(\prod_{k = 1}^n A_k \right) C_n,
\end{equation}
where $A_k, C_n$ are i.i.d. copies of $A$ and $C$, independent of each other.
This is identical to \eqref{eqn:sumjump2} by setting $A_n = \beta_n$ and $C_n = J_n$.

The following corollary of the fact that a SSA gamma process is associated with an exponential generic jump, giving a representation of a gamma distributed random variable, is also given in \cite[example 3.8.2]{vervaat1979stochastic}.

\begin{corollary}[Vervaat \cite{vervaat1979stochastic}]\label{thm:vervaat}
    If $C \sim$ exp$(\lambda)$ and $A \sim$ beta$(\theta, 1)$ for some $\lambda, \theta > 0$, then the unique solution to the stochastic difference equation \eqref{eqn:vervaatsde} is $X \sim$ gamma$(\theta, \lambda)$.
\end{corollary}

However, Vervaat did not make any connection with SSA processes in his work on stochastic difference equations, as he did not treat $\left(\prod_{k = 1}^n A_k \right)$ as a time index, and he did not discuss selfdecomposability of $X$, only infinitely divisibility.

\paragraph{Extremal processes}

As mentioned in Section \ref{sec:intro}, the $1$-SSA exponential process with \ssarate $\theta = 1$ arises in the theory of extremal process introduced by Dwass \cite{dwass1964extremal}.

Starting from an i.i.d. sequence $(X_n, n = 1, 2, \ldots)$ of continuous random variables, it is elementary that the {\em record sequence} $(M_n := \max_{1 \le k \le n } X_k, n = 1,2, \ldots)$ is a Markov chain with state space $\R$ and transition probabilities specified by 
\begin{equation}\label{eqn:mncdf}
    \P( M_{n+m} \le y \mid M_n = x) = F^m (y) 1(x \le y),
\end{equation}
where $F^t(y)$ is the $t$-th power of the common cumulative distribution function (c.d.f.) of the $X_i$.
It was observed in 1964 by Dwass \cite{dwass1964extremal} and Lamperti \cite{lamperti1964extreme} that the record process can be generalized to a time-homogeneous pure jump-type Markov process, known as the {\em extremal process} $M := (M(t), t > 0)$, described by its entrance law
\begin{equation}\label{eqn:entrancelaw}
    \P( M(t) \le y ) = F^t(y),
\end{equation}
and the hold-jump description: conditional on $M(s) = x$,
\begin{itemize}
    \item `Hold' - at level $x$ for an exponential time $H_x$ with rate $Q(x):= - \log F(x)$
    so that
    \begin{equation}\label{holds}
        \P( H_x > t ) = \P( M(t) \le x ) = F^t(x) =  e^{- t Q(x) } ;
    \end{equation}
    \item `Jump' - at time $H_x$ to a state $L_x:= M({s + H_x})$ with distribution
    \begin{equation}\label{jx}
        \P( L_x > b ) = \frac{ Q(b) }{Q(x) } \qquad ( b \ge x ).
    \end{equation}
\end{itemize}
See Shorrock \cite{shorrock1974discrete} or Kallenberg \cite[Chapter 13]{kallenberg2021foundations} for more details.

\begin{proposition}[Dwass \cite{dwass1964extremal} Resnick and Rubinovitch \cite{resnick1973structure}
Shorrock \cite{shorrock1974discrete} ]\label{thm:poissonjumps}
    Let $((T_z, M_z), z \in \Z)$ be a listing indexed by integers $\Z$ of the times $T_z$ of jumps of $(M(t), t \ge 0)$  
    and the corresponding record levels $M_z := M({T_z})$, with $T_z < T_{z+1}$. Then 
    \begin{equation}\label{mtzpoints}
        range ( (M_z, T_{z+1} - T_z) , z \in \Z ) \mbox{ is } PPP( Q(dm) e^{- Q(m) t } dt ).
    \end{equation}
    In particular,
    \begin{itemize}
        \item[(I)] the random set of record times $T_z$ is a scale invariant PPP on $\R_+$ with rate $1$;
        \item[(II)] the random set of record levels $range ( M_z, z \in \Z )$ is $PPP( Q(dm) /Q(m))$;
        \item[(III)] the random set of holding times at these record levels $range ( T_{z+1} - T_z, z \in \Z )$ is a scale invariant PPP on $\R_+$ with rate $1$.
    \end{itemize}
\end{proposition}

As indicated by Dwass and Lamperti, the extremal Markov process associated with each continuous distribution $F$ on the line is essentially the same as that associated with every other continuous distribution $F'$, via the monotonic transformation
\begin{equation}\label{cdftransform}
    M' = \psi(M), \quad \text{for } \psi \text{ with } F'(\cdot) = F( \psi \in \cdot ).
\end{equation}

To illustrate the point and see its relation with our Theorems \ref{thm:GammatoPPP} and \ref{thm:ssarate}, consider the case of Gnedenko's extreme value distribution, that is the distribution of $\eps^{-1}$ for $\eps$ exponential with mean $1$:
\begin{equation}\label{eqn:gnedenkoF}
    F(y) = \exp( - y^{-1} )  = \P( \eps^{-1} \le y ) \qquad ( y > 0 ),
\end{equation}
hence the rate in \eqref{holds} is
\begin{equation} \label{eqn:gnedenkoQ}
    Q(x) = -\log F(x) = \frac 1x,
\end{equation}
and the jump has distribution
\begin{equation}
    \P( J_x > b ) = \frac{ Q(b) }{Q(x) } = \frac xb \qquad ( b \ge x ),
\end{equation}
namely $x / J_x \sim$ beta$(1, 1)$.
It is easy to check that the functional inverse (which means the `hold' and `jump' are swapped) coincides with the hold-jump description \eqref{gammasrec} \eqref{gammatrec} of a $1$-SSA gamma process $\Tsvar$ with \ssarate $1$ and $T(s) \ed s\eps \sim$ exp$(s^{-1})$.

Therefore, the path of an extremal process associated with \eqref{eqn:gnedenkoF} has the same distribution as the functional inverse of the path of a standard $1$-SSA exponential process. This matches Proposition \ref{thm:poissonjumps}(I) with Theorem \ref{thm:GammatoPPP}, and Proposition \ref{thm:poissonjumps}(III) with Theorem \ref{thm:ssarate}(III). Lastly, in this case the $PPP( Q(dm) /Q(m))$ in Proposition \ref{thm:poissonjumps}(II) is a scale invariant PPP on $\R_+$ with rate $1$, as in Theorem \ref{thm:ssarate}(II).

It may also be of interest to replace the i.i.d. sequence above with an inhomogeneous Markov sequence $(X_n, n = 1, 2, \ldots)$ defined as follows:
\begin{itemize}
    \item The first term $X_1$ has c.d.f. $F^\theta$;
    \item Conditionally on $M_n = x$, let $F_{x+}$ be the c.d.f. of $X$ given $X>x$
    \begin{equation}
        F_{x+}(y) := (F(y) - F(x))/(1 - F(x)) \qquad ( y > x ),
    \end{equation}
    and $F_{x-}$be the one given $X \le x$
    \begin{equation}
        F_{x-}(y) := F(y) / F(x) \qquad ( y \le x ).
    \end{equation}
    Then the distribution of $X_{n+1}$ is the mixture with weights $1 - F(x)$ and $F(x)$ of the distribution of a random variable with c.d.f. $F_{x+}^\theta$, and distribution of a random variable with c.d.f. $F_{x-}$, i.e.
    \begin{equation} \label{otherinhomo}
        \P(X_{n+1} \le y \mid M_n = x) = (1 - F(x)) F_{x+}^\theta(y) + F(x) F_{x-}(y).
    \end{equation}
\end{itemize}

An example of this sequence is obtained by putting $X \sim$ Uniform$(0, 1)$. Then $X_1 \sim$ beta$(1, \theta)$ and conditionally on $M_n = x$, $X_{n+1}$ is the mixture with weights $1-x$ and $x$ of Uniform$(0, 1-x)$ and $x + (1-x) \beta$ for $\beta \sim$ beta$(1, \theta)$.

\begin{theorem} \label{thm:discreteinhomo}
    Suppose $X$ is a real-valued random variable without atoms, and the sequence $(X_n, n = 1, 2, \ldots)$ follows the inductive construction above.
    Then the maximum indicators $(B_n, n = 1, 2, \ldots)$ form a sequence of independent Bernoulli random variables with $\E B_n = \theta / (\theta + n - 1)$.
\end{theorem} 

This is a generalization to Najnudel and Pitman \cite[Corollary 1.4]{najnudel2019feller} where they proved the same result for $X \sim$ Uniform($0, 1$).
This generalization works because the distribution of $X$ has no atoms, hence has no influence on record times.

It is not hard to construct a time-inhomogeneous extremal process $M$ associated with $(X_n, n = 1, 2, \ldots)$ and check that the path of such extremal process has the same distribution as the functional inverse of the path of a $1$-SSA gamma process $\Tsvar$ with $T(1)\sim$ gamma$(\theta, 1)$.
In this case, Theorem \ref{thm:discreteinhomo} can be read as a discrete-time analogue of Theorem \ref{thm:GammatoPPP}, and an analogue of Proposition \ref{thm:poissonjumps} can also be easily given by replacing the rate in (I) and (III) by $\theta$, and replacing the PPP in (II) by $PPP(\theta ~ Q(dm) /Q(m))$.

\paragraph{Scale invariant point processes}

Scale invariant point processes, and scale invariant random sets, including the scale invariant PPP, were 
studied in \cite{pitman1996random,gnedin2005self} by considering the random partitions of $(0, 1)$ related
to the Poisson-Dirichlet distribution.
It was observed in \cite{pitman1996random} that a random closed set is scale invariant if and only if the associated age process is $1$-self-similar, while \cite{gnedin2005self} remarked on the relationship between scale invariant PPP and records.

\paragraph{The scale invariant Poisson spacings lemma for general $\theta \neq 1$}

The formulation of the Theorem \ref{thm:GammatoPPP} for general $\theta >0$ was suggested by proofs of the scale invariant Poisson spacings Theorem \ref{thm:sips}, first indicated by Arratia \cite{arratia1998central, arratia2002amount}, and detailed by Arratia, Barbour and Tavar\'{e} \cite{arratia2006tale}, where they make use of a Poisson point process on $(0, \infty)^2$ which turns out to be our $\eta$ in \eqref{eqn:SSAPPPrep}.
Based on \cite{arratia2006tale}, Gnedin \cite[Section 4]{gnedin2008corners} pointed out that the scale invariant Poisson spacings theorem for positive integer $\theta$ follows from a specialization of Ignatov's theorem in the form of \cite[Corollary 5.1]{goldie1984record}.

\paragraph{Feller's coupling and the Ewens sampling formula}

The parameter $\theta$ in the Poisson spacings theorem is related to the Ewens$(\theta)$ distribution \cite{ewens1972sampling} as a generalization to the uniform random permutation of $[n]$. Feller \cite{feller1945fundamental} provided a coupling between the counts of cycles of various sizes in a uniform random permutation of $[n]$ and the spacings between successes in a sequence of $n$ independent Bernoulli$(k^{-1})$ trials at the $k$th trial.
Informally, this sequence of independent Bernoulli trials is a discrete analogue of the scale invariant PPP with rate $1$ relative to $x^{-1} dx$.

Ignatov \cite{ignatov1981point} proved that in an infinite sequence of independent Bernoulli$(n^{-1})$ trials, as the indicators of record values in an i.i.d. sequence, the numbers of spacings of length $k$ between successes/records are independent Poisson variables with means $k^{-1}$. 
This Poisson sequence provides another discrete analogue of the scale invariant PPP with rate $1$.
It is interesting that this discrete result was obtained many years after theory of extremal processes by Dwass \cite{dwass1964extremal}.
Ignatov's result was generalized by Arratia, Barbour and Tavar\'e \cite{arratia2006tale} in the study of cycles of (non-uniform) random permutations governed by the Ewens$(\theta)$ distribution, i.e. a permutation is weighted $\theta ^ k$ if there are $k$ cycles.
See Najnudel and Pitman \cite{najnudel2019feller} for details of the coupling between the random permutations governed by the Ewens$(\theta)$ distribution and the sequence of inhomogeneous Bernoulli$(\theta(\theta + n - 1) ^{-1})$ trials, as mentioned in Theorem \ref{thm:discreteinhomo}.

\bibliographystyle{acm} 
\bibliography{SelfSimilar}

\end{document}